\documentclass[reqno,11pt]{amsart}
\usepackage{amsmath}
\usepackage{amsfonts}
\usepackage{amssymb}
\usepackage{stmaryrd}
\usepackage{graphicx}
\usepackage{color}
\usepackage{cite}
\usepackage{times}
\usepackage{tikz}
\usepackage{pgfplots}
\pgfplotsset{compat=1.10}
\usepgfplotslibrary{fillbetween}
\usetikzlibrary{patterns}

\newtheorem{theorem}{Theorem}[section]
\newtheorem{corollary}[theorem]{Corollary}
\newtheorem{lemma}[theorem]{Lemma}
\newtheorem{proposition}[theorem]{Proposition}

\newtheorem{remark}[theorem]{Remark}

\DeclareMathAlphabet{\mathpzc}{OT1}{pzc}{m}{it}

\newcommand{\rd}{\mathrm{d}}
\definecolor{cadmiumgreen}{rgb}{0.0, 0.42, 0.24}
\numberwithin{equation}{section}
\title[]{Stationary solutions to a nonlocal fourth-order\\ elliptic obstacle problem}

\author{Philippe Lauren\c{c}ot}
\address{Institut de Math\'ematiques de Toulouse, UMR~5219, Universit\'e de Toulouse, CNRS \\ F--31062 Toulouse Cedex 9, France}
\email{laurenco@math.univ-toulouse.fr}
\thanks{Partially supported by the CNRS Projet International de Coop\'eration Scientifique PICS07710}

\author{Christoph Walker}
\address{Leibniz Universit\"at Hannover\\ Institut f\" ur Angewandte Mathematik \\ Welfengarten 1 \\ D--30167 Hannover\\ Germany}
\email{walker@ifam.uni-hannover.de}

\begin{document}

\date{\today}

\begin{abstract}
 Existence of stationary solutions to a nonlocal fourth-order elliptic obstacle problem arising from the modelling of microelectromechanical systems with heterogeneous dielectric properties is shown. The underlying variational structure of the model is exploited to construct these solutions as minimizers of a suitably regularized energy, which allows us to weaken considerably the assumptions on the model used in a previous article.
\end{abstract}


\maketitle

\section{Introduction}

Idealized electrostatically actuated microelectromechanical systems (MEMS) are made up of an elastic conducting plate which is clamped on its boundary and suspended above a rigid conducting ground plate. Their dynamics results from the competition between mechanical and electrostatic forces in which the elastic plate is deformed by a Coulomb force induced by holding the two plates at different electrostatic potentials. When the electrostatic forces dominate the mechanical ones, the elastic plate comes into contact with the ground plate, thereby generating a short circuit and leading to the occurrence of a touchdown singularity in the related mathematical models, see \cite{EGG10, FMPS06, GPW05, LWBible, Pel01a, PeB03} and the references therein. However, covering the ground plate with a thin insulating layer prevents a direct contact of the two plates and, from a mathematical point of view, features a constraint of obstacle-type which hinders the touchdown singularity. Different models have been developed to take into account the influence of the coating layer deposited on the ground plate, most of them relying on the so-called small aspect ratio approximation and describing the state of the MEMS device by the sole deformation of the elastic plate \cite{AmEtal, BG01, LLG14, LLG15, YZZ12}. A more elaborate model is derived in \cite[Section~5]{LW19}, in which the state of the device is not only given by the deformation of the elastic plate, but also by the electrostatic potential in the region between the two plates.

To give a more precise account, we restrict ourselves to a two-dimensional setting, neglecting variations in the transverse horizontal direction, so that the geometry of the device under study herein is the following, see Figure~\ref{Fig3}. At rest, the cross-section  of the elastic plate is $D:=(-L,L)$, $L>0$, and it is clamped at its boundary $(x,z)=(\pm L,0)$. The fixed ground plate has the same shape $D$ and is located at $z=-H-d$. It is coated with an insulating layer $$\Omega_1 := D \times (-H-d,-H)$$ of thickness $d>0$  with \textit{a priori} non-uniform dielectric permittivity $\sigma_1>0$ and which cannot be penetrated by the elastic plate. {As a consequence, the vertical displacement $u:\bar D \rightarrow \mathbb{R}$ of the elastic plate actually ranges in $[-H,\infty)$ and the contact region $\{ (x,-H)\ :\ x\in D\,, \ u(x)=-H\}$ between the insulating layer and the elastic plate  might be non-empty. We assume also that the free space 
\begin{equation*}
 \Omega_2(u):=\left\{(x,z)\in D\times \mathbb{R}\,:\, -H<  z <  u(x)\right\}
\end{equation*} 
between the upper part of the insulating layer and the elastic plate has uniform permittivity $\sigma_2>0$, and we denote the electrostatic potential in the device
\begin{equation*}
\Omega({u}):=\left\{(x,z)\in D\times \mathbb{R} \,:\, -H-d<  z <  u(x)\right\}=\Omega_1\cup  \Omega_2( {u})\cup \Sigma(u)\,,
\end{equation*}
by $\psi_u$, where $\Sigma(u)$ is the interface
\begin{equation*}
\Sigma(u):=\{(x,-H)\,:\, x\in D,\, u(x)>-H\}\,. 
\end{equation*}
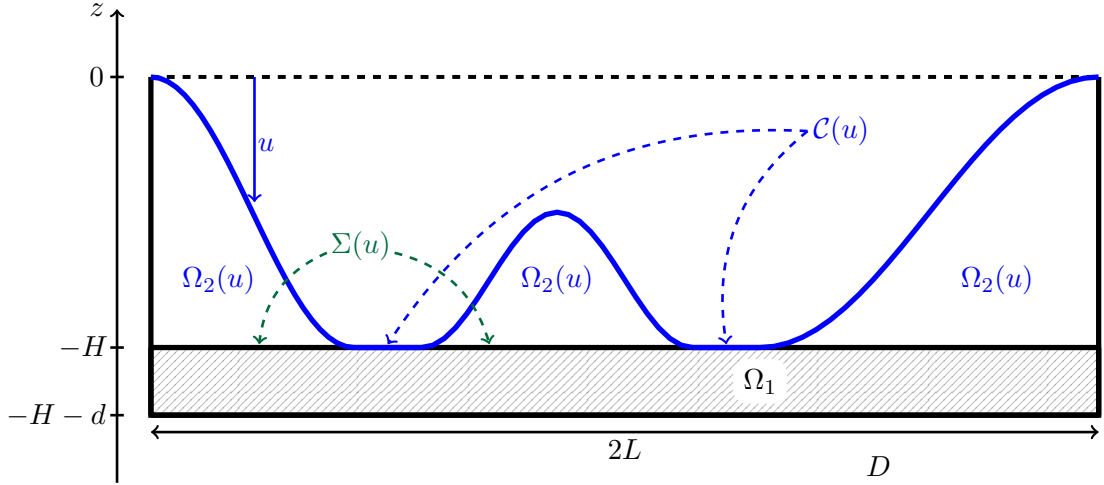
\begin{figure}
	\begin{tikzpicture}[scale=0.9]
	\draw[black, line width = 1.5pt, dashed] (-7,0)--(7,0);
	\draw[black, line width = 2pt] (-7,0)--(-7,-5);
	\draw[black, line width = 2pt] (7,-5)--(7,0);
	\draw[black, line width = 2pt] (-7,-5)--(7,-5);
	\draw[black, line width = 2pt] (-7,-4)--(7,-4);
	\draw[black, line width = 2pt, fill=gray, pattern = north east lines, fill opacity = 0.5] (-7,-4)--(-7,-5)--(7,-5)--(7,-4);
	\draw[blue, line width = 2pt] plot[domain=-7:-4] (\x,{-2-2*cos((pi*(\x+4)/3) r)});
	\draw[blue, line width = 2pt] (-4,-4)--(-3,-4);
	\draw[blue, line width = 2pt] plot[domain=-3:1] (\x,{-3+cos((pi*(\x+1)/2) r)});
	\draw[blue, line width = 2pt] (1,-4)--(2,-4);
	\draw[blue, line width = 2pt] plot[domain=2:7] (\x,{-2-2*cos((pi*(\x-2)/5) r)});
	\draw[blue, line width = 1pt, arrows=->] (-5.47,0)--(-5.47,-1.85);
	\node at (-5.3,-1) {${\color{blue} u}$};
	\node[draw,rectangle,white,fill=white, rounded corners=5pt] at (2,-4.5) {$\Omega_1$};
	\node at (2,-4.5) {$\Omega_1$};
	\node at (-6,-3) {${\color{blue} \Omega_2(u)}$};
	\node at (-1,-3) {${\color{blue}\Omega_2(u)}$};
	\node at (5.5,-3) {${\color{blue} \Omega_2(u)}$};
	\node at (3.75,-5.75) {$D$};
	\node at (-7.8,1) {$z$};
	\draw[black, line width = 1pt, arrows = ->] (-7.5,-6)--(-7.5,1);
	\node at (-8.4,-5) {$-H-d$};
	\draw[black, line width = 1pt] (-7.6,-5)--(-7.4,-5);
	\node at (-8,-4) {$-H$};
	\draw[black, line width = 1pt] (-7.6,-4)--(-7.4,-4);
	\node at (-7.8,0) {$0$};
	\draw[black, line width = 1pt] (-7.6,0)--(-7.4,0);
	\node at (0,-5.5) {$2L$};
	\draw[black, line width = 1pt, arrows = <->] (-7,-5.25)--(7,-5.25);
	\node at (3.2,-0.8) {${\color{blue} \mathcal{C}(u)}$};
	\draw (2.7,-0.8) edge[->,bend right,blue, dashed, line width = 1pt] (1.5,-3.95);
	\draw (2.7,-0.8) edge[->,bend right, blue, dashed, line width = 1pt] (-3.5,-3.95);
	\node at (-3.9,-2.5) {${\color{cadmiumgreen} \Sigma(u)}$};
	\draw (-4.4,-2.6) edge[->,bend right, cadmiumgreen, dashed, line width = 1pt] (-5.4,-3.95);
	\draw (-3.4,-2.6) edge[->,bend left, cadmiumgreen, dashed, line width = 1pt] (-2,-3.95);
	\end{tikzpicture}
	\caption{Geometry of $\Omega(u)$ for a state $u\in \bar{S}_0$ with non-empty (and disconnected) coincidence set $\mathcal{C}(u)$.}\label{Fig3}
\end{figure}
According to the model derived in \cite[Section~5]{LW19}, equilibrium configurations of the above described device are weak solutions $u\in H_D^2(D)$} to the fourth-order  obstacle problem
\begin{equation}
\beta\partial_x^4 u-(\tau+\alpha\|\partial_x u\|_{L_2(D)}^2)\partial_x^2 u+\partial\mathbb{I}_{\bar{S}_0}(u) \ni -g(u) \;\;\text{ in }\;\; D\,, \label{bennygoodman}
\end{equation}
where 
\begin{equation*}
H_D^2(D):= \left\{ v\in H^2(D)\,:\, v(\pm L) = \partial_x v(\pm L)=0 \right\} \,,
\end{equation*}
and $\partial\mathbb{I}_{\bar{S}_0}(u)$ denotes the subdifferential in $H_D^2(D)$ of the indicator function $\mathbb{I}_{\bar{S}_0}$ of the closed convex subset 
\begin{equation*}
\bar{S}_0 := \left\{ v\in H_D^2(D)\,:\,  v\ge -H \text{ in } D \right\}
\end{equation*}
of $H_D^2(D)$. We recall that, given $v\in \bar{S}_0$, the subdifferential $\partial\mathbb{I}_{\bar{S}_0}(v)$ is the subset of the dual space
\begin{equation*}
H^{-2}(D):=\big(H_D^2(D)\big)'
\end{equation*}
of $H_D^2(D)$ given by 
\begin{equation*}
\partial\mathbb{I}_{\bar{S}_0}(v) :=\left\{ \xi\in H^{-2}(D)\,:\, \langle \xi , v-w \rangle_{H_D^2} \ge 0\,, \, w\in \bar{S}_0\right\}\,,
\end{equation*}
where $\langle \cdot , \cdot \rangle_{H_D^2}$ denotes the duality pairing between $H^{-2}(D)$ and $H_D^2(D)$. If $v\not\in \bar{S}_0$, then $\partial\mathbb{I}_{\bar{S}_0}(v):=\emptyset$.  While $\partial\mathbb{I}_{\bar{S}_0}(u)$ accounts for the non-penetrability of the insulating layer, the fourth- and second-order terms in \eqref{bennygoodman} represent forces due to plate bending and plate stretching, respectively. These forces are balanced by the electrostatic force $g(u)$ acting on the elastic plate, which is derived in \cite{LW19} and involves the electrostatic potential $\psi_u$ in the device. The latter solves the transmission problem
\begin{subequations}\label{TMP}
	\begin{align}
	\mathrm{div}(\sigma\nabla\psi_u)&=0 \quad\text{in }\ \Omega(u)\,, \label{TMP1}\\
	\llbracket \psi_u \rrbracket = \llbracket \sigma \partial_z \psi_u \rrbracket &=0\quad\text{on }\ \Sigma(u)\,, \label{TMP2}\\
	\psi_u&=h_u\quad\text{on }\ \partial\Omega(u)\,, \label{TMP3}
	\end{align}   
\end{subequations}
in the domain $\Omega(u)$, see Figure~\ref{Fig3}. In \eqref{TMP}, $\llbracket\cdot\rrbracket$ denotes the jump across the interface
$\Sigma(u)$, the dielectric permittivity $\sigma$ is given by
\begin{equation*}
\sigma(x,z) := \left\{
\begin{array}{lcl}
\sigma_1(x,z) & \text{ for } & (x,z)\in \Omega_1\,, \\
& & \vspace{-4mm}\\
\sigma_2 & \text{ for } & (x,z)\in D\times (-H,\infty)\,,
\end{array}
\right.
\end{equation*}
with 
\begin{subequations}\label{bobbybrown}
\begin{equation}\label{S}
\sigma_1\in C^2\big( \overline{\Omega}_1 \big)\,, \quad \min_{\overline{\Omega}_1} \sigma_1>0 \,,\qquad \sigma_2\in (0,\infty)\,,
\end{equation}
and the non-homogeneous Dirichlet boundary conditions $h_u$ are given by 
\begin{equation*}
h_u(x,z):= h(x,z,u(x)) = \left\{ \begin{array}{ll}
h_1(x,z,u(x))\,, & (x,z)\in \overline{\Omega}_1\,, \\
h_2(x,z,u(x))\,, & (x,z)\in \overline{\Omega_2(u)}\,,
\end{array} \right.
\end{equation*}
where
\begin{equation}\label{bobbybrown2a}
h_1: \bar{D}\times [-H-d,-H]\times [-H,\infty)\rightarrow [0,\infty)
\end{equation}
and 
\begin{equation}\label{bobbybrown2aa}
h_2: \bar{D}\times [-H,\infty)\times [-H,\infty)\rightarrow [0,\infty)
\end{equation}
are $C^2$-smooth functions satisfying
\begin{align}
h_1(x,-H,w)&=h_2(x,-H,w)\,,\quad (x,w)\in D\times [-H,\infty)\,,\label{bobbybrown2}\\
\sigma_1(x,-H)\partial_z h_1(x,-H,w)& =\sigma_2\partial_z h_2(x,-H,w)\,,\quad (x,w)\in D\times [-H,\infty)\,.\label{bobbybrown3}
\end{align}
\end{subequations}
We note that \eqref{bobbybrown2}-\eqref{bobbybrown3} imply that $h_u$ satisfies
\begin{equation}\label{hh}
\llbracket h_u \rrbracket = \llbracket \sigma \partial_z h_u \rrbracket =0\quad \text{on }\ \Sigma(u)
\end{equation} 
and thus complies with the transmission conditions \eqref{TMP2}, see \cite[Example~5.5]{LW19} for an example of functions $h_1$ and $h_2$ satisfying the above assumptions. With these assumptions, the electrostatic force $g(u)$ is computed in \cite[Theorems~1.2 and~1.3]{LW19}. It has a different expression at contact points between the plates and at points where the elastic plate is strictly above the insulating layer. Specifically, introducing the coincidence set 
\begin{equation*}
\mathcal{C}(u):=\{x\in D\,:\, u(x)=-H\}
\end{equation*}
for $u\in \bar{S}_0$, the electrostatic force is given by
\begin{subequations}\label{G}
\begin{equation}\label{g}
\begin{split}
g(u)(x) :=\ & \mathfrak{g}(u)(x) -\frac{\sigma_2}{2}  \left[ \big((\partial_x h_2)_u\big)^2+ \big((\partial_z h_2)_u+(\partial_w h_2)_u\big)^2 \right](x, u(x))\\
&  + \left[ \sigma_1(\partial_w h_1)_u\,\partial_z\psi_{u,1} \right](x,-H-d)
\end{split}
\end{equation}
for $x\in D$, where
\begin{equation}\label{ggg}
\mathfrak{g}(u)(x):=\left\{\begin{array}{ll}  \displaystyle{\frac{\sigma_2}{2} \big(1+(\partial_x u(x))^2\big)\,\big[\partial_z\psi_{u,2}-(\partial_z h_2)_u-(\partial_w h_2)_u\big]^2(x, u(x))} \,, & x\in D\setminus \mathcal{C}(u)\,,\\
\hphantom{x}\vspace{-3.5mm}\\
\displaystyle{\frac{\sigma_2}{2}\, \left[ \frac{\sigma_1}{\sigma_2}\partial_z\psi_{u,1}-(\partial_z h_2)_u-(\partial_w h_2)_u \right]^2(x, -H)}\,,  & x\in  \mathcal{C}(u)\,,
\end{array}\right.
\end{equation}
\end{subequations}
and $(\psi_{u,1},\psi_{u,2}) := \left( \psi_u \mathbf{1}_{\overline{\Omega}_1} , \psi_u \mathbf{1}_{\overline{\Omega_2(u)}}\right)$. It is readily seen from \eqref{G} that $g(u)$ features a nonlinear and nonlocal dependence on $u$, the latter being due to the terms involving $\psi_u$ in \eqref{G}.

\bigskip

The investigation of the solvability of \eqref{bennygoodman}-\eqref{TMP} is initiated in \cite{LW19}, exploiting the variational structure underlying the derivation of \eqref{bennygoodman}-\eqref{TMP} which implies that solutions to \eqref{bennygoodman} are critical points in $\bar{S}_0$ of an energy functional $E$, which is actually the total energy of the device. Specifically,  
\begin{subequations}\label{Energy}
	\begin{equation}
	E(u) := E_m(u)  + E_e(u) \label{E}
	\end{equation}
	consists of the mechanical energy 
	\begin{equation}\label{Em}
	E_m(u):=\frac{\beta}{2} \|\partial_x^2u\|_{L_2(D)}^2 +\left( \frac{\tau}{2} + \frac{\alpha}{4}\|\partial_x u\|_{L_2(D)}^2 \right) \|\partial_x u\|_{L_2(D)}^2
	\end{equation}
	and the electrostatic energy 
	\begin{equation}\label{Ee}
	E_e(u) := -\frac{1}{2} \int_{\Omega(u)} \sigma \vert\nabla \psi_u\vert^2\,\rd (x,z)\,.
	\end{equation}
\end{subequations}
 Then \eqref{bennygoodman} subject to \eqref{G} is the Euler-Lagrange equation for minimizers of $E$ in $\bar{S}_0$, see \cite{LW19}, and $g(u)$ defined in \eqref{G} corresponds to the (directional) derivative of $E_e(u)$ with respect to~$u$, see \cite[Theorem~1.4]{LW19}. The existence of solutions to \eqref{bennygoodman}  is established in \cite[Section~5]{LW19} by showing that the energy functional has at least a  minimizer on $\bar{S}_0$. This, however, requires additional assumptions ensuring that the electrostatic energy $E_e(u)$ does not grow faster than $\|u\|_{H^1(D)}^2$ as well as the coercivity of the energy functional $E$. More precisely, to guarantee the former (see \eqref{gEe} below) we assume that there are constants $m_i>0$, $i=1,2,3$, such that
\begin{subequations}\label{bobby}
\begin{equation}
\vert \partial_x h_1(x,z,w)\vert +\vert\partial_z h_1(x,z,w)\vert  \le \sqrt{m_1+m_2 w^2}\,, \quad \vert\partial_w h_1(x,z,w)\vert \le \sqrt{m_3}\,, \label{bobbybrown5}
\end{equation}
for $(x,z,w)\in \bar D \times [-H-d,-H] \times [-H,\infty)$ and
\begin{equation} 
\vert \partial_x h_2(x,z,w)\vert +\vert\partial_z h_2(x,z,w)\vert \le \sqrt{\frac{m_1+m_2 w^2}{H+w}}\,,\quad \vert\partial_w h_2(x,z,w)\vert \le \sqrt{\frac{m_3}{H+w}}\,,\label{bobbybrown6}
\end{equation}
\end{subequations}
for $(x,z,w)\in \bar D \times [-H,\infty) \times [-H,\infty)$. The existence result from \cite{LW19} then reads:

\begin{proposition}\cite[Theorems~5.1 and~5.3]{LW19}\label{T0}
Let $\beta>0$ and $\tau,\alpha\ge 0$. Assume that \eqref{bobbybrown}, \eqref{bobby}, \eqref{bobbybrown1x}, and \eqref{bobbybrown1y} hold,  and that the ground plate and the elastic plate are kept at constant, but different, electrostatic potentials; that is, there is $V>0$ such that
\begin{align}
h_1(x,-H-d,w)& =0\,,\quad (x,w)\in \bar{D}\times [-H,\infty)\,,\label{bobbybrown1x}\\
h_2(x,w,w)&=V\,,\quad (x,w)\in \bar{D}\times [-H,\infty)\,.\label{bobbybrown1y}
\end{align}
If
\begin{equation}\label{bbb}
\max\{\alpha,\mathfrak{K}\}>0
\end{equation}
with $$
\mathfrak{K} :=\beta - 4L^2 \left[ (d+1) \max\{ \|\sigma_1\|_{L_\infty(\Omega_1)} , \sigma_2\} \left( 12 m_2L^2+2 m_3 \right)- \tau \right]_+ \,,
$$
then  there is at least one solution $u\in\bar{S}_0$ to the variational inequality \eqref{bennygoodman}  in the following sense: for all $w\in \bar{S}_0$, 
$$
\int_D \left\{\beta\partial_x^2 u\,\partial_x^2 (w-u)+\big[\tau+\alpha\|\partial_x u\|_{L_2(D)}^2\big]\partial_x u\, \partial_x(w-u)\right\}\,\rd x\ge -\int_D g(u) (w-u)\, \rd x  
$$
where $g(u)$ is given by \eqref{G} and $\psi_u$ is the solution to \eqref{TMP}. Here, we interpret $\partial_x^4 u$ for $u\in \bar{S}_0$ as an element of $H^{-2}(D)$ by virtue of
$$
\langle \partial_x^4 u,\phi\rangle_{H_D^2}:=\int_D \partial_x^2 u \partial_x^2\phi\,\rd x\,,\quad \phi\in H_D^2(D)\,.
$$
Moreover, this solution can be obtained as a minimizer of $E$ on $\bar{S}_0$.
\end{proposition}

Obviously, a first step towards a full proof of Proposition~\ref{T0} is to solve the transmission problem~\eqref{TMP} for the electrostatic potential $\psi_u$ with sufficient regularity in order to give a meaning to the function $g$ along with deriving suitable continuity properties. We refer to  Section~\ref{Sec2.1} for a detailed account on this issue (see in particular Lemma~\ref{L6} and Lemma~\ref{L2} where these results are recalled). A second essential step in the proof of Proposition~\ref{T0} consists of deriving the coercivity of the energy functional $E$ on~$\bar{S}_0$. This property is ensured by assumption~\eqref{bbb} (along with  \eqref{bobby}). In particular, if $\alpha>0$, then the mechanical energy $E_m$ involves a super-quadratic term which allows a compensation of the negative contribution from the electrostatic energy $E_e$.

 \begin{remark}
Note that \eqref{bobbybrown1x} implies
$$
\partial_w h_1(x,-H-d,w) =0\,,\quad (x,w)\in \bar{D}\times [-H,\infty)\,,
$$
while \eqref{bobbybrown1y} implies
\begin{equation*}
\partial_x h_2(x,w,w)=0\,,\quad (x,w)\in \bar{D}\times [-H,\infty)\,,
\end{equation*}
and
\begin{equation*}
\partial_z h_2(x,w,w) + \partial_w h_2(x,w,w)=0\,,\quad (x,w)\in \bar{D}\times [-H,\infty)\,,
\end{equation*}
so that the formula \eqref{G} for $g(u)$ simplifies and becomes $g(u) = \mathfrak{g}(u)$ in \eqref{g}. In particular, the function $g(u)$ is non-negative.
\end{remark}

The aim of the present work is to establish the existence of a solution to \eqref{bennygoodman} under considerably weaker assumptions. In particular, we shall get rid of the technical and somewhat artificial assumption~\eqref{bbb}. Since \eqref{bbb} is obviously satisfied when $\alpha>0$, we shall treat $\alpha$ as zero in the following computations.
Moreover, we no longer need a sign for the function $g(u)$ and can slightly weaken assumption~\eqref{bobbybrown1y}. Indeed, we only require that
\begin{subequations}\label{KK}
\begin{equation}\label{Kbound0}
\partial_w h_1(x,-H-d,w) =0\,,\quad (x,w)\in \bar{D}\times [-H,\infty)\,,
\end{equation}
and that there is a number $K>0$ such that
\begin{equation}\label{Kbound}
\vert \partial_x h_2(x,w,w)\vert + \vert \partial_z h_2(x,w,w)+\partial_w h_2(x,w,w)\vert \le K\,,\quad (x,w)\in \bar D\times [-H,\infty)\,.
\end{equation}
\end{subequations}
Clearly, \eqref{bobbybrown1x}-\eqref{bobbybrown1y} imply \eqref{KK}. Also note that, due to \eqref{Kbound0}, the last term in the definition of $g(u)$ in \eqref{g} vanishes, i.e. $g(u)$ reduces to
\begin{equation}\label{g2}
\begin{split}
g(u)(x) =\ & \mathfrak{g}(u)(x) -\frac{\sigma_2}{2}  \left[ \big((\partial_x h_2)_u\big)^2+ \big((\partial_z h_2)_u+(\partial_w h_2)_u\big)^2 \right](x, u(x))\,.
\end{split}
\end{equation}
With these assumptions we can now formulate the main result of the present paper.

\begin{theorem}\label{T1}
Let $\beta>0$, $\tau\ge 0$, and $\alpha= 0$. Assume that \eqref{bobbybrown}, \eqref{bobby}, and \eqref{KK} hold.
 Then  there is at least one solution $u\in\bar{S}_0$ to the variational inequality \eqref{bennygoodman}  in the sense of Proposition~\ref{T0}. More precisely, the functional $E$ is bounded from below on $\bar{S}_0$ and has a minimizer on $\bar{S}_0$ which is a weak solution to \eqref{bennygoodman}.
\end{theorem}

Since we no longer impose assumption~\eqref{bbb} in Theorem~\ref{T1}, the boundedness from below of the functional $E$ is {\it a priori} unclear, due to the negative contribution from the electrostatic energy $E_e$. We thus shall work with regularized coercive functionals instead (see \eqref{Ek} below) and use comparison principle arguments to derive {\it a priori} bounds on minimizers of the regularized functionals, see Section~\ref{Sec2} below. We shall then prove that cluster points of these minimizers are actually minimizers of the original functional $E$. The full proof of Theorem~\ref{T1} is given in Section~\ref{Sec2} and relies on an idea introduced previously in a related work \cite{LNW20}. 

\begin{remark}\label{rem:apos}
Theorem~\ref{T1} remains valid when $\alpha>0$ and then only requires the assumptions \eqref{bobbybrown} and \eqref{bobby}. Indeed, the existence of a minimizer of $E$ on $\bar{S}_0$ is shown as in the proof of \cite[Theorem~5.1]{LW19} and this minimizer is a weak solution to \eqref{bennygoodman} as a consequence of \cite[Theorem~5.3]{LW19}.  
\end{remark}

Finally, we provide an additional property of weak solutions to \eqref{bennygoodman} when the potentials applied on the elastic plate and the ground plate are constant.

\begin{corollary}\label{C11}
Suppose \eqref{bobbybrown}, \eqref{bobbybrown1x}, and \eqref{bobbybrown1y}. If the coincidence set $\mathcal{C}(u)\subset D$ of a solution $u\in\bar{S}_0$ to \eqref{bennygoodman} is  non-empty, then it is an interval.
\end{corollary}

\section{Proof of Theorem~\ref{T1} and Corollary~\ref{C11}}\label{Sec2}

\subsection{Auxiliary Results}\label{Sec2.1}

Let us emphasize that the function $g(u)$ defined in \eqref{G} involves gradient traces of the electrostatic potential $\psi_u$, the latter solving the transmission problem~\eqref{TMP} posed on the non-smooth domain $\Omega(u)$ which possesses corners. In addition, $\Omega_2(u)$ need not be connected, but may consist of several components with non-Lipschitz boundaries (see Figure~\ref{Fig3}), so that traces have first to be given a meaning. While the existence of a unique variational solution $\psi_u\in h_u+H_0^1(\Omega(u))$ to \eqref{TMP} readily follows from the Lax-Milgram theorem, the required further regularity for $\psi_u$ in order to make sense of its gradient traces is thus far from being obvious. Moreover, $\psi_u$ (and hence $g(u)$) depends non-locally on~$u$ so that continuity properties with respect to the plate deformation~$u$ is non-trivial.

Nevertheless, relying on shape optimization methods and Gamma convergence techniques the following result regarding the existence to a solution of the transmission problem~\eqref{TMP} and its regularity  is shown in \cite{LW19}.

\begin{lemma} \label{L6}
Suppose \eqref{bobbybrown}. For each $u\in \bar{S}_0$, there is a unique variational solution $\psi_u \in h_u + H^1_0(\Omega(u))$ to \eqref{TMP}.  Moreover, 
\begin{equation*}
\psi_{u,1} =  \psi_u \mathbf{1}_{\overline{\Omega}_1} \in H^2(\Omega_1)\,,\qquad \psi_{u,2} = \psi_u \mathbf{1}_{\overline{\Omega_2(u)}} \in H^2(\Omega_2(u))\,,
\end{equation*} 
and $\psi_{u}$  is a strong solution to the transmission problem~\eqref{TMP} satisfying $\sigma\partial_z \psi_u\in H^1(\Omega(u))$.  Also, 
\begin{equation*}
\inf_{\partial\Omega(u)} h_u \le \psi_u(x,z) \le \sup_{\partial\Omega(u)} h_u\,, \quad (x,z)\in \overline{\Omega(u)}\,.
\end{equation*}
\end{lemma}

\begin{proof}
The existence, uniqueness, and regularity of the variational solution $\psi_u$ to \eqref{TMP} follows from \cite[Theorem~1.1]{LW19}, while the upper and lower bounds for $\psi_u$ are consequences of the weak maximum principle \cite[Chapter~7, Exercice~2.2]{Necas67}, since $\sigma\in L_\infty(\Omega(u))$.
\end{proof}

The regularity of $\psi_u$ provided by Lemma~\ref{L6} in particular guarantees that $g(u)$ defined in \eqref{G} is meaningful for $u\in \bar{S}_0$. As for the continuity of $g(u)$  with respect to $u\in \bar{S}_0$ we recall:
 
\begin{lemma}\label{L2}
Suppose \eqref{bobbybrown}. 
\begin{itemize}
\item[{\bf (a)}] The mapping $g:\bar{S}_0\rightarrow L_2(D)$ is well-defined, continuous, and bounded on bounded sets, the set $\bar{S}_0$ being endowed with the topology of $H^2(D)$.
\item[{\bf (b)}] Let $(u_j)_{j\ge 1}$ be a sequence of functions in $\bar{S}_0$ such that $u_j\rightharpoonup u$ in $H^2(D)$ for some $u\in \bar{S}_0$. Then
\begin{equation*}
\lim_{j\to\infty} \|g(u_j)-g(u)\|_{ L_2(D)} = 0 \quad\text{ and }\quad \lim_{j\to\infty} E_e(u_j) = E_e(u)\,.
\end{equation*}
\end{itemize}
\end{lemma}

\begin{proof}
Part {\bf (a)} follows from \cite[Theorem~1.4]{LW19}, from \cite[Corollary~3.14 \& Lemma~3.16]{LW19}, and the continuity of the trace operator from $H^1(\Omega_1)$ to $L_p(D\times\{-H\})$ for all $p\in [1,\infty)$.

Part~{\bf (b)} follows from \cite[Proposition~3.17 \& Corollary~3.12]{LW19}.
\end{proof}

\subsection{Minimizers for a Regularized Energy}

In the following we let $\beta>0$ and $\tau\ge 0$ and assume throughout that \eqref{bobbybrown}, \eqref{bobby}, and \eqref{KK} hold. We put
$$
\bar\sigma:=\max\left\{\|\sigma_1\|_{L_\infty(\Omega_1)}\,,\,\sigma_2\right\}\,.
$$
In order to prove Theorem~\ref{T1} it suffices to find a minimizer of the energy functional $E$ on~$\bar{S}_0$ since any such minimizer satisfies \eqref{bennygoodman} according to \cite[Theorem~5.3]{LW19} (note that \cite[Theorem~5.3]{LW19} obviously remains true without imposing \cite[Assumption~(5.2a)]{LW19}). However, as mentioned previously, the  coercivity of the energy functional $E$ is \textit{a priori} unclear when dropping assumption~\eqref{bbb}. For this reason we introduce for $k \ge 0$ the regularized functional 
\begin{equation}\label{Ek}
\mathcal{E}_k(u):= E(u) + \frac{A}{2} \|(u-k)_+\|_{L_2(D)}^2\,,\quad u\in \bar{S}_0\,,
\end{equation}
where $E(u)$ is defined in \eqref{Energy} and the constant $A$ given by
$$
A:=8 (d+1) \bar{\sigma} \left(\frac{3m_2}{2}+\frac{m_3^2 (d+1)  \bar{\sigma}}{\beta}\right) 
$$
with constants $m_j$ introduced in \eqref{bobby}. We shall now prove, for each $k>0$, the existence of a minimizer $u_k$ of $\mathcal{E}_k$ on $\bar{S}_0$ and subsequently derive an {\it a priori} bound on such minimizers, so that the additional regularizing term drops out in $\mathcal{E}_k$. We first  show the coercivity of the functional $\mathcal{E}_k$.

\begin{lemma}\label{L333}
Given $k\ge H$, there is a constant $c(k)>0$ such that
$$
\mathcal{E}_k(u)\ge \frac{\beta}{4}\|\partial_x^2 u\|_{L_2(D)}^2+\frac{A}{4}\|(u-k)_+\|_{L_2(D)}^2- c(k)\,,\quad u\in \bar{S}_0\,.
$$
\end{lemma}

\begin{proof}
Let $u\in \bar{S}_0$. The variational characterization of $\psi_u$, see \cite[Lemma~3.2]{LW19}, readily gives
\begin{equation*}
\begin{split}
-E_e(u)&=\frac{1}{2}\int_{\Omega(u)} \sigma \vert \nabla \psi_u\vert^2\, \rd (x,z) \le \frac{1}{2}\int_{\Omega(u)} \sigma \vert \nabla h_u\vert^2\, \rd (x,z) \,.
\end{split}
\end{equation*}
Thus, invoking \eqref{bobbybrown},  Young's inequality, and the definition of $\Omega(u)$ we derive
\begin{align}
-E_e(u)&\le \int_{\Omega(u)} \sigma \left[( \partial_x h(x,z,u(x))^2+(\partial_w h(x,z,u(x))^2(\partial_x u)^2\right]\, \rd (x,z) \nonumber \\
& \quad +\frac{1}{2}\int_{\Omega(u)} \sigma (\partial_z h(x,z,u(x))^2\, \rd (x,z) \nonumber \\
&\le (d+1)   \bar{\sigma} \left\{\frac{3}{2} m_1 \vert D\vert+\frac{3}{2} m_2\|u\|_{L_2(D)}^2+m_3\|\partial_x u\|_{L_2(D)}^2\right\}\,. \label{gEe}
\end{align}
Next, since $u\in \bar{S}_0\subset H_D^2(D)$ implies
$$
\int_D \vert\partial_x u\vert^2\,\rd x=-\int_D u\partial_x^2u\,\rd x\le \|u\|_{L_2(D)} \|\partial_x^2u\|_{L_2(D)}
$$
we deduce from Young's inequality
\begin{equation*}
\begin{split}
-E_e(u)&\le  (d+1)  \bar{\sigma} \left\{\frac{3}{2} m_1 \vert D\vert+\frac{3}{2} m_2\|u\|_{L_2(D)}^2+m_3\|u\|_{L_2(D)} \|\partial_x^2u\|_{L_2(D)}\right\}\\
& \le  (d+1)  \bar{\sigma} \left\{\frac{3}{2} m_1 \vert D\vert+\left(\frac{3m_2}{2}+\frac{m_3^2 (d+1)  \bar{\sigma} }{\beta}\right)\|u\|_{L_2(D)}^2\right\} \\
& \qquad +\frac{\beta}{4} \|\partial_x^2u\|_{L_2(D)}^2\,.
\end{split}
\end{equation*}
Finally, note that
\begin{align*}
\|u\|_{L_2(D)}^2 & = \int_D u^2 \mathbf{1}_{[-H,k]}(u)\,\rd x + \int_D u^2 \mathbf{1}_{(k,\infty)}(u)\, \rd x \\
& \le k^2 \int_D \mathbf{1}_{[-H,k]}(u)\,\rd x +2  \int_D (u-k)^2 \mathbf{1}_{(k,\infty)}(u)\, \rd x + 2 k^2 \int_D \mathbf{1}_{(k,\infty)}(u)\, \rd x \\
& \le 2 k^2 |D| + 2 \|(u-k)_+\|_{L_2(D)}^2 \,.
\end{align*}
Hence, taking  the previous two inequalities into account,  the definition of $\mathcal{E}_k(u)$ entails
\begin{equation*}
\begin{split}
\mathcal{E}_k(u)&\ge  \frac{\beta}{4} \|\partial_x^2u\|_{L_2(D)}^2   + \frac{A}{2}\|(u-k)_+\|_{L_2(D)}^2 - \frac{3}{2} (d+1)  \bar{\sigma} m_1 \vert D\vert\\
& \qquad -  \frac{A}{8} \left(2\|(u-k)_+\|_{L_2(D)}^2+ 2k^2 |D| \right) \\
&\ge 
\frac{\beta}{4} \|\partial_x^2u\|_{L_2(D)}^2 + \frac{A}{4} \|(u-k)_+\|_{L_2(D)}^2 - c(k)
\end{split}
\end{equation*}
as claimed.
\end{proof}

The just established coercivity now easily yields the existence of a minimizer of $\mathcal{E}_k$ on $\bar{S}_0$. 

\begin{proposition}\label{L334}
For each $k\ge H$, the functional $\mathcal{E}_k$ has at least one minimizer $u_k\in \bar{S}_0$ on $\bar{S}_0$; that is,
\begin{equation}\label{mini}
\mathcal{E}_k(u_k)=\min_{\bar{S}_0} \mathcal{E}_k\,.
\end{equation}
Moreover, $u_k \in \bar{S}_0$ is a weak solution to the variational inequality
\begin{equation}
\beta\partial_x^4 u_k-\tau \partial_x^2 u_k+
A (u_k-k)_++\partial\mathbb{I}_{\bar{S}_0}(u_k) \ni -g(u_k) \;\;\text{ in }\;\; D\,. \label{bennygoodman2.0}
\end{equation}
\end{proposition}

\begin{proof}
Clearly, $E_m$ defined in \eqref{Em} is weakly lower semicontinuous on $H^2(D)$ while $E_e$ is continuous with respect to the weak topology of $H^2(D)$ due to Lemma~\ref{L2}. Thanks to Lemma~\ref{L333}, the direct method of the calculus of variations now easily yields the existence of a minimizer  $u_k\in \bar{S}_0$ of $\mathcal{E}_k$ on  $\bar{S}_0$.
In particular,
$$
0\le \liminf_{s\rightarrow 0^+} \frac{1}{ s}\big(\mathcal{E}_k(u_k+s(w-u_k))-\mathcal{E}_k(u_k)\big)
$$
for any fixed
\begin{equation*}
w \in S_0 :=  \left\{ v\in H_D^2(D)\,:\,  v> -H \text{ in } D \right\} \subset \bar{S}_0\,.
\end{equation*}
It was shown in \cite[Theorem~1.4]{LW19} that (since $u_k+s(w-u_k)\in S_0$ for $s\in (0,1)$)
\begin{equation*}
\lim_{s\rightarrow 0^+} \frac{1}{s}\big(E_e(u_k+s(w-u_k))-E_e(u_k)\big)= \int_D g(u_k) (w-u_k)\, \rd x\,.
\end{equation*}
From the definition of $\mathcal{E}_k$ we then obtain by gathering the two limits that
\begin{align*}
 \int_D &\left\{\beta\partial_x^2 u_k\,\partial_x^2 (w-u_k) +  \tau  \partial_x u_k\, \partial_x (w-u_k) + A(u_k-k)_+ (w-u_k)\right\}\,\rd x \\
& \qquad\qquad \ge - \int_D g(u_k) (w-u_k)\, \rd x
\end{align*}
for all $w\in S_0$. Since $S_0$ is dense in $\bar{S}_0$,  this inequality also holds for any $w\in \bar{S}_0$,  and we thus have shown that $u_k$ satisfies the variational formulation of~\eqref{bennygoodman2.0}.
\end{proof}

\subsection{\textit{A Priori} Bounds}

We shall now show that $u_k$ is {\it a priori} bounded for $k$ large enough (making the additional term in $\mathcal{E}_k$ superfluous). To this aim we need an {\it a priori} bound for the solution to the fourth-order boundary value problem \eqref{eA3.1} subject to suitable  Dirichlet boundary conditions as stated below. The bound relies on the maximum principle for the fourth order operator $\beta\partial_x^4-\tau\partial_x^2$ with clamped boundary conditions \cite{Bo05,Gr02,LW15,Ow97}.

\begin{lemma}\label{propA3.1}
Let $G_0\ge 0$ and recall that $\beta>0$ and $\tau\ge 0$. For an interval $I:=(a,b)\subset (-L, L)$, let $S_I \in C^4([a,b])$ denote the unique solution to the boundary-value problem
\begin{equation}
\beta S_I'''' - \tau S_I'' = G_0\ , \qquad x\in (a,b)\ , \label{eA3.1}
\end{equation}
supplemented with one of the boundary conditions: 
\begin{align}
& S_I(a) + H = S_I'(a)= S_I(b) + H = S_I'(b) = 0 \quad\text { if }\quad -L<a<b<L\ , \label{eA3.2}\\
& S_I(-L) = S_I'(-L)= S_I(b) + H = S_I'(b) = 0 \quad\text { if }\quad -L=a<b<L\ , \label{eA3.3} \\
& S_I(a)+H = S_I'(a)= S_I(L) = S_I'(L) = 0 \quad\text { if }\quad -L<a<b=L\ , \label{eA3.4} \\
& S_I(-L) = S_I'(-L)= S_I(L) = S_I'(L) = 0 \quad\text { if }\quad -L=a<b=L\ . \label{eA3.5}
\end{align}
Then, there is $\kappa_0>0$ depending only on $G_0$, $\beta$, $L$, $H$, and $\tau$ (but not on $I=(a,b)$) such that 
\begin{equation*}
\left| S_I(x) \right| \le \kappa_0\ , \quad x\in [a,b]\, .
\end{equation*}
\end{lemma}

\begin{proof} This result has already been observed in \cite[Lemma~A.1]{LNW20} and we include its proof only for the sake of completeness. Note that \eqref{eA3.1} subject to one of the boundary conditions \eqref{eA3.2}-\eqref{eA3.5} indeed admits a unique solution $S_I$.

\noindent\textbf{Case~1: $-L<a<b<L$.} Set $P(y) := S_I(a+(b-a)y)+H$ for $y\in [0,1]$ and note that $P$ solves the boundary-value problem 
\begin{equation}
\begin{split}
& \beta P'''' - \tau (b-a)^2 P'' = (b-a)^4 G_0\ , \qquad y\in (0,1)\ , \\
& P(0) = P'(0) = P(1) = P'(1) = 0\ . 
\end{split} \label{eA3.6}
\end{equation}
Since $G_0\ge 0$ we deduce that $P\ge 0$ in $(0,1)$ from a version of Boggio's comparison principle \cite{Bo05, Gr02, LW15, Ow97} . 
Testing \eqref{eA3.6} by $P$ we get
\begin{equation*}
\beta \|P''\|_{L_2(0,1)}^2 + \tau (b-a)^2 \|P'\|_{L_2(0,1)}^2 = (b-a)^4 G_0 \int_0^1 P(y)\ \rd y \ .
\end{equation*}
Since
\begin{equation*}
|P(y)| = \left| \int_0^y (y-y_*) P''(y_*) \ \rd y_* \right|  \le \| P''\|_{L_2(0,1)}\ ,
\end{equation*}
we infer from the above inequalities that
\begin{equation*}
\beta \|P\|_{L_\infty(0,1)}^2 \le \beta \|P''\|_{L_2(0,1)}^2  \le (b-a)^4 G_0 \| P\|_{L_\infty(0,1)} \le 16L^4 G_0 \| P\|_{L_\infty(0,1)}\ .
\end{equation*}
Consequently, $0\le P \le 16L^4 G_0/ \beta$ in $[0,1]$, hence $-H \le S_I \le 16L^4 G_0/ \beta - H$ in $[a,b]$.

\medskip

\noindent\textbf{Case~2: $-L=a<b<L$.} Define $Q(y):= y^2(y^2 +2(H-1)y +1- 3H)$ for $y\in [0,1]$ and note that  $Q(0)=Q'(0) = Q(1)+ H = Q'(1)= 0$. Set $P(y) := S_I(-L+(b+L)y)-Q(y)$ for $y\in [0,1]$. Then, due to \eqref{eA3.1} and \eqref{eA3.3}, $P$ solves the boundary-value problem 
\begin{equation}
\begin{split}
& \beta P'''' - \tau (b+L)^2 P'' = (b+L)^4 G_0 - \beta Q'''' + \tau (b+L)^2 Q'' \ , \qquad y\in (0,1)\ , \\
& P(0) = P'(0) = P(1) = P'(1) = 0\ . 
\end{split} \label{eA3.7}
\end{equation}
The arguments of Case~1 give
\begin{align*}
\beta \|P\|_{L_\infty(0,1)}^2 & \le \beta \|P''\|_{L_2(0,1)}^2 \\ 
&  \le \left[ (b+L)^4 G_0+  \beta \|Q''''\|_{L_\infty(0,1)}+ \tau (b+L)^2\|Q''\|_{L_\infty(0,1)}\right] \|P\|_{L_\infty(0,1)}\\
& \le \left[ (b+L)^4 G_0 + 24 \beta +14 \tau (H+1) (b+L)^2 \right] \|P\|_{L_\infty(0,1)} \\
& \le \left[ 16L^4 G_0 +  24 \beta + 56 \tau (H+1) L^2 \right] \|P\|_{L_\infty(0,1)}\ ,
\end{align*}
since $Q''''=24$,  $|Q''|\le  14 (H+1)$, and $b<L$. Therefore, 
\begin{equation*}
\begin{split}
\|S_I \|_{L_\infty(-L,b)} & \le \|P\|_{L_\infty(0,1)} + \|Q\|_{L_\infty(0,1)} \\
& \le \frac{16L^4 G_0 +24 \beta + 56 \tau (H+1) L^2}{\beta} + \|Q\|_{L_\infty(0,1)}\ .
\end{split}
\end{equation*}
\medskip

\noindent\textbf{Case~3: $-L<a<b=L$.} Define $P(y) := S_I(a+y(L-a))-Q(1-y)$ for $y\in [0,1]$, where $Q$ is as in {\bf Case~2}. Arguing as in the previous case we obtain the same bound for $\|S_I\|_{L_\infty(a,L)}$.
\medskip

\noindent\textbf{Case~4: $-L=a<b=L$.} Define $P(y) := S_I(-L+2Ly)$ for $y\in [0,1]$. Then, by \eqref{eA3.1} and \eqref{eA3.5}, $P$ solves the boundary-value problem 
\begin{equation*}
\begin{split}
& \beta P'''' - 4\tau L^2 P'' = 16L^4 G_0\ , \qquad y\in (0,1)\ , \\
& P(0) = P'(0) = P(1) = P'(1) = 0\ . 
\end{split} 
\end{equation*}
As in Case~1 we deduce that $0\le S_I \le 16L^4 G_0/\beta$ in $[-L,L]$.
\end{proof}

The previous lemma now implies the desired {\it a priori} bounds on the minimizers $u_k$.

\begin{proposition}\label{prop4}
There is $\kappa_0 \ge H>0$ depending only on $K$ introduced in \eqref{Kbound} such that, for all $k\ge H$, the minimizer $u_k\in \bar{S}_0$ of $\mathcal{E}_k$ on $\bar{S}_0$ constructed in Proposition~\ref{L334} satisfies $\|u_k\|_{L_\infty(D)}\le \kappa_0$.
\end{proposition}

\begin{proof}  Let $k\ge H$. We first note that, since $\mathfrak{g}(u_k)\ge 0$ in $D$ by \eqref{ggg},  it easily follows from \eqref{Kbound} and \eqref{g2} that
\begin{equation}\label{bono}
g(u_k)(x)\ge -G_0:= -\sigma_2 K^2\,, \quad x\in D\,.
\end{equation}
Next, since $u_k\in C(\bar{D})$ with $u_k(\pm L)=~0$, the set $\{ x \in D\,:\, u_k(x)>-H\}$
 is a non-empty open subset of $D$. Owing to \cite[IX.Pro\-po\-si\-tion~1.8]{AEIII} we can thus write it as a countable union of open intervals $(I_j)_{j\in J}$. Consider a fixed $j\in J$ and let $S_{I_j}$ denote the solution to \eqref{eA3.1} in $I_j$ subject to the associated boundary conditions on $\partial I_j$ listed in \eqref{eA3.2}-\eqref{eA3.5},  which vary according to whether $\bar{I}_j\subset D$ or not. Then Lemma~\ref{propA3.1} yields a constant $\kappa_0>H$ (independent of $j\in J$)  such that
\begin{equation}\label{Sb}
\|S_{I_j}\|_{L_\infty(I_j)}\le \kappa_0\,.
\end{equation} 
Note that, by definition of $I_j$, the function $u_k$ restricted to $I_j$ satisfies the same boundary conditions on $\partial I_j$ as $S_{I_j}$. Hence, for $z:=u_k-S_{I_j}\in H^2(I_j)$ we have $z=\partial_x z=0$ on $\partial I_j$. Moreover, if $\theta\in \mathcal{D}(I_j)$, then $u_k\pm \delta\theta\in \bar{S}_0$ for $\delta>0$ small enough since $u_k>-H$ in the support of $\theta$. Invoking  the weak formulation of \eqref{bennygoodman2.0} we derive
\begin{equation*}
\pm\delta\int_{I_j} \left\{\beta\partial_x^2 u_k\,\partial_x^2 \theta+\tau \partial_x u_k\, \partial_x\theta  +A(u_k-k)_+ \theta\right\}\,\rd x \ge \mp\delta\int_{I_j} g(u_k) \theta\, \rd x  \,,
\end{equation*}
hence
\begin{equation*}
\int_{I_j} \left\{\beta\partial_x^2 u_k\,\partial_x^2 \theta+\tau \partial_x u_k\, \partial_x\theta  +A(u_k-k)_+ \theta\right\}\,\rd x =-\int_{I_j} g(u_k) \theta\, \rd x  \,.
\end{equation*}
Thus,  we conclude that $z=u_k-S_{I_j} \in H^2(I_j)$ weakly solves the boundary value problem
\begin{subequations}\label{ep}
\begin{align}
\beta\partial_x^4 z-\tau\partial_x^2 z& = -G_0-g(u_k)-A (u_k-k)_+\;\;\text{ in }\;\; I_j\,,\label{ep1}\\
  z&=\partial_x z=0 \;\;\text{ on }\;\; \partial I_j\,. \label{ep2}
\end{align}
\end{subequations}
Now, since $g(u_k)+A (u_k-k)_+\in L_2(I_j)$ by Lemma~\ref{L2} {\bf (a)}, classical elliptic regularity theory \cite{GGS10} entails that $z=u_k-S_{I_j}\in H^4(I_j)$ is a strong solution to \eqref{ep}. Furthermore, it follows from \eqref{bono} and the non-negativity of $A(u_k-k)_+$ that the right-hand side of \eqref{ep1} is a non-positive function. Boggio's comparison principle \cite{Bo05, Gr02, LW15, Ow97} then implies that $z=u_k-S_{I_j} < 0$ in $I_j$. Consequently, we infer from \eqref{Sb} that
$\|u_k\|_{L_\infty(I_j)}\le \kappa_0$. Since $\kappa_0$ is independent of $j\in J$ and $(I_j)_{j\in J}$ covers $\{ x \in D\,:\, u_k(x)>-H\}$, the assertion follows.
\end{proof}

\subsection{Proof of Theorem~\ref{T1}}

We are now in a position to finish off the proof of Theorem~\ref{T1}. Indeed, if $u_k\in \bar{S}_0$ is the minimizer of the functional $\mathcal{E}_k$ on $\bar{S}_0$ provided by Proposition~\ref{L334}, then $-H\le u_k\le \kappa_0$ in $D$ due to Proposition~\ref{prop4}. Consequently,  for $k\ge \kappa_0$ we have
\begin{equation}\label{888}
E(u_k)=\mathcal{E}_{ \kappa_0}(u_k)=\mathcal{E}_k(u_k)\le \mathcal{E}_k(v) = E(v) + \frac{A}{2} \|(v-k)_+\|_{L_2(D)}^2 \,,\quad v\in \bar{S}_0\,.
\end{equation}
Now, since $0\in \bar{S}_0$ it follows from Lemma~\ref{L333} that, for $k\ge  \kappa_0$,
$$
\frac{\beta}{4}\|\partial_x^2 u_k\|_{L_2(D)}^2 \le \mathcal{E}_{\kappa_0}(u_k)+ c(\kappa_0)\le \mathcal{E}_{ \kappa_0}(0)+ c(\kappa_0)=E(0)+ c(\kappa_0)\,.
$$
Thus,  $(u_k)_{k\ge  \kappa_0}$ is bounded in $H^2(D)$ so that there is a subsequence (not relabeled) converging weakly in $H^2(D)$ and strongly in $H^1(D)$ towards some $u_*\in \bar{S}_0$. Since $E_m$ is  weakly lower semicontinuous on $H^2(D)$ and since $E_e$ is continuous with respect to the weak topology of $H^2(D)$ owing to Lemma~\ref{L2}~{\bf (b)}, we obtain from \eqref{888} that
\begin{equation*}
E(u_*)\le E(v)\,,\quad v\in \bar{S}_0\,,
\end{equation*}
recalling that the continuous embedding of $H^1(D)$ in $C(\bar{D})$ readily implies that
\begin{equation*}
\lim_{k\to\infty} \|(v-k)_+\|_{L_2(D)} = 0\,,\quad v\in \bar{S}_0\,.
\end{equation*}
Therefore, $u_*\in \bar{S}_0$ minimizes $E$ on $\bar{S}_0$. Now \cite[Theorem~5.3]{LW19} entails that $u_*$  satisfies the variational inequality \eqref{bennygoodman}. Alternatively, one can use the weak convergence in $H^2(D)$ and the strong convergence in $H^1(D)$ of $(u_k)_{k\ge \kappa_0}$ to $u_*$ to pass to the limit $k\to\infty$ in \eqref{bennygoodman2.0}, observing that $(g(u_k))_{k\ge\kappa_0}$ converges to $g(u)$ in $L_2(D)$ by Lemma~\ref{L2}~{\bf (b)} and that $(u_k-k)_+=0$ for $k\ge \kappa_0$. This proves Theorem~\ref{T1}. \qed

\bigskip

\subsection{Proof of Corollary~\ref{C11}}

Suppose \eqref{bobbybrown}, \eqref{bobbybrown1x}, and \eqref{bobbybrown1y} and let $u\in\bar{S}_0$ be any solution to the variational inequality \eqref{bennygoodman}. Note that \eqref{bobbybrown1x} and \eqref{bobbybrown1y} imply  $g(u) = \mathfrak{g}(u)$ in \eqref{g}. In particular, the function $g(u)$ is non-negative. Assume now for contradiction that the coincidence set $\mathcal{C}(u)$ is not an interval. Then there are  $-L<a<b<L$ with  $u(a)+H=\partial_x u(a)=u(b)+H=\partial_x u(b)=0$ and $u>-H$ in $I:=(a,b)$. 
We may then argue as in the proof of Proposition~\ref{prop4} to conclude that $z:=u+H\in H^4(I)\cap H_D^2(I)$ is a strong solution to the boundary value problem
\begin{subequations}\label{epq}
\begin{align}
\beta\partial_x^4 z-\tau\partial_x^2 z& = -g(u)\;\;\text{ in }\;\; I\,,\label{ep1q}\\
  z&=\partial_x z=0 \;\;\text{ on }\;\; \partial I\,. \label{ep2q}
\end{align}
\end{subequations}
Another application of a version of Boggio's comparison principle \cite{Bo05, Gr02, LW15, Ow97} implies $z=u+H \le  0$ in $I$. But this contradicts $u>-H$ in $I$. \qed

\bibliographystyle{siam}
\bibliography{BG_Transmission_Model_2}

\end{document}